\documentclass{article}

\usepackage{amssymb,amsthm,amsmath}

\title{Gaussian approximation of moments of sums of independent symmetric 
random variables with logarithmically concave tails\thanks{Partially supported by the Foundation for Polish Science}}
\author{Rafa{\l} Lata{\l}a
\thanks{Institute of Mathematics, 
University of Warsaw, 
Banacha 2, 
02-097 Warszawa, 
Poland
and
Institute of Mathematics,
Polish Academy of Sciences,
\'Sniadeckich 8,
P.O.Box 21, 00-956 Warszawa 10,
Poland, email: {\tt rlatala@mimuw.edu.pl}
}}
\date{}

\newtheorem{lem}{Lemma}
\newtheorem{cor}{Corollary}
\newtheorem{thm}{Theorem}
\newtheorem{prop}{Proposition}

\def\Ex{{\mathbf E}}
\def\Pr{{\mathbf P}}
\def\er{{\mathbb R}}
\def\ve{\varepsilon}
\def\cale{\mathcal E}

\begin{document}

\maketitle

\begin{abstract}
We study how well moments of sums of independent symmetric random variables
with logarithmically concave tails may be approximated by moments of Gaussian
random variables.
\end{abstract}

Let $\ve_1,\ve_2,\ldots$ be a Bernoulli sequence, i.e.\
a sequence of independent symmetric variables taking values $\pm 1$. Hitczenko \cite{H}
showed that for $p\geq 2$ and $S=\sum_{i}a_i\ve_i$,
\begin{equation}
\label{radmom}
\|S\|_p\sim \sum_{i\leq p}a_i^*+\sqrt{p}\Big(\sum_{i>p} (a_{i}^*)^2\Big)^{1/2}
\end{equation}
where $(a_i^*)$ denotes the nonincreasing rearrangement of $(|a_i|)$ and 
$f(p)\sim g(p)$ means that there exists a universal constant $C$ such that 
$C^{-1}f(p)\leq g(p)\leq Cf(p)$ for any parameter $p$ (see also \cite{MS} and \cite{HK}
for related results). Gluskin and Kwapie\'n \cite{GK} generalized the result 
of Hitczenko and found two sided bounds for moments of sums of independent symmetric
random variables with logarithmically concave tails (we say that $X$ has logarithmically
concave tails if $\ln\Pr(|X|\geq t)$ is concave from $[0,\infty)$ to $[-\infty,0]$).
In particular they showed that
for a sequence $(\cale_i)$ of independent symmetric exponential random variables
with variance 1 (i.e. the density $2^{-1/2}\exp(-\sqrt{2}|x|)$), $S=\sum_i a_i\cale_i$, 
and $p\geq 2$,
\begin{equation}
\label{expmom}
\|S\|_p\sim p\|a\|_{\infty}+\sqrt{p}\|a\|_2,
\end{equation}
where $\|a\|_p=(\sum_i |a_i|^p)^{1/p}$ for $1\leq p<\infty$ and 
$\|a\|_{\infty}=\sup|a_i|$.
Two sided inequality for moments of sums of arbitrary independent symmetric random variables
was derived in \cite{La}.

Results (\ref{radmom}) and (\ref{expmom}) suggest that if all coefficients are of order $o(1/p)$ then $\|S\|_p$
should be close to the $p$-th norm of the corresponding Gaussian sum that is to $\gamma_p\|a\|_2$,
where $\gamma_p=\|{\cal N}(0,1)\|_p=2^{p/2}\Gamma(\frac{p+1}{2})/\sqrt{\pi}$. 
The purpose of our note is to verify this assertion.

First we show the intuitive result that in the class of normalized symmetric random variables with logarithmically concave tails Bernoulli and exponential random variables
are extremal.

\begin{prop}
\label{extr}
Let $X_i$ be independent symmetric r.v.'s with logarithmically concave tails such that
$\Ex X_i^2=1$. Then for any
$p\geq 3$,
\[
\Big\|\sum_{i=1}^n a_i \ve_i\Big\|_p\leq \Big\|\sum_{i=1}^n a_i X_i\Big\|_p
\leq \Big\|\sum_{i=1}^n a_i \cale_i\Big\|_p.
\]
\end{prop}

\begin{proof}
Lower bound follows from Theorem 1.1 of \cite{5A} (in fact we do not use here the assumption of logconcavity
of tails). To prove the upper bound it is enough to show that for all $a,b\in \er$ and
$p\geq 3$,
\[
\Ex |a+bX_i|^p\leq \Ex|a+b\cale_i|^p.
\]
Let $\varphi(x)=\frac{1}{2}(|a+bx|^p+|a-bx|^p)$, then $\varphi'$ is convex on $[0,\infty)$ with 
$\varphi'(0)=0$. Since $\Ex X_i^2=1=\Ex\cale_i^2$ there exist $t_0$ such that
$\Pr(|X_i|\geq t_0)=\Pr(|\cale_i|\geq t_0)$. Logconcavity of tails implies
that $\Pr(|X_i|\geq t)\leq \Pr(|\cale_i|\geq t)$ for $t\geq t_0$ and the opposite inequality
holds for $0\leq t\leq t_0$. Let $\varphi'(t_0)=ct_0$ for some $c>0$. Then by convexity
of $\varphi'$ we have $(\varphi'(t)-ct)(\Pr(|\cale_i|\geq t)-\Pr(|X_i|\geq t))\geq 0$ for all
$t$. Thus
\begin{align*}
0&\leq \int_{0}^{\infty}(\varphi'(t)-ct)(\Pr(|\cale_i|\geq t)-\Pr(|X_i|\geq t))dt
\\
&=\Ex(\varphi(\cale_i)-\varphi(X_i))-\frac{c}{2}\Ex(\cale_i^2-X_i^2)
=\Ex|a+b\cale_i|^p-\Ex|a+bX_i|^p.
\end{align*}
\end{proof}

Next technical lemma will be used to compare characteristic functions of Bernoulli and
exponential sums.

\begin{lem}
Let $|a_1|\geq |a_2|\geq \ldots \geq|a_n|$. Then for any $t$,
\begin{equation}
\label{est1}
\prod_{i=1}^n\cos(a_i t)+\frac{1}{2}a_{1}^2t^2\geq \prod_{i=2}^{n}\frac{1}{1+a_{i}^2t^2/2}.
\end{equation}
\end{lem}

\begin{proof}
We will consider 3 cases.

{\bf Case I} $|a_1 t|\leq \sqrt{2}$. Let $x_i=a_i^2t^2/2$, then since $\cos(a_i t)\geq 1-a_i^2t^2/2\geq 0$, to establish (\ref{est1}) it is enough to show that
\[
\prod_{i=1}^n(1-x_i)+x_1\geq \prod_{i=2}^n\frac{1}{1+x_i}
\mbox{ for } 1\geq x_1\geq x_2\geq \ldots\geq x_n\geq 0. 
\]
However,
\begin{align*}
\prod_{i=2}^n(1+x_i)&\Big[\prod_{i=1}^n(1-x_i)+x_1\Big]
=(1-x_1)\prod_{i=2}^n(1-x_i^2)+x_1\prod_{i=2}^n(1+x_i)
\\
&\geq (1-x_1)(1-\sum_{i=2}^nx_i^2)+x_1(1+\sum_{i=2}^n x_i)
\geq 1-\sum_{i=2}^n x_i^2+\sum_{i=2}^nx_1x_i\geq 1.
\end{align*}

{\bf Case II} $\sqrt{2}\leq |a_1 t|\leq \pi/2$. Then
\[
\prod_{i=1}^n\cos(a_i t)+\frac{1}{2}a_{1}^2t^2\geq \frac{1}{2}a_{i}^2t^2
\geq 1 \geq \prod_{i=2}^{n}\frac{1}{1+a_{i}^2t^2/2}.
\]

{\bf Case III} $|a_1 t|\geq \pi/2$. Then
\[
\prod_{i=1}^n\cos(a_i t)+\frac{1}{2}a_{1}^2t^2\geq \frac{1}{2}a_{i}^2t^2-|\cos(a_1 t)|
\geq 1 \geq \prod_{i=2}^{n}\frac{1}{1+a_{i}^2t^2/2}.
\] 
\end{proof}

Using the above lemma we may now compare moments of Bernoulli and exponential sums
in the special case $p\in [2,4]$.

\begin{lem}
\label{p_2_4}
Let $|a_1|\geq |a_2|\geq \ldots \geq |a_n|$. Then for any $2\leq p\leq 4$,
\begin{equation}
\label{comp1}
\Ex\Big|\sum_{i=1}^n a_{i}\ve_i\Big|^{p}\geq \Ex\Big|\sum_{i=2}^n a_i\cale_i\Big|^{p}. 
\end{equation}
\end{lem}

\begin{proof}
Let $S_1=\sum_{i=1}^n a_{i}\ve_i$ and $S_2=\sum_{i=2}^n a_i\cale_i$, obviously
we may assume that $2<p<4$.
By Lemma 4.2 of \cite{Ha} we have for any random variable $X$ with finite fourth moment,
\[
\Ex |X|^p=C_p\int_{0}^{\infty}\Big(\varphi_X(t)-1+\frac{1}{2}t^2\Ex|X|^2\Big)t^{-p-1}dt.
\]
where $\varphi_X$ is the characteristic function of $X$ and $C_p=-\frac{2}{\pi}\sin(\frac{p\pi}{2})\Gamma(p+1)>0$. Notice that by Lemma 1, 
\[
\varphi_{S_1}(t)-\varphi_{S_2}(t) =
\prod_{i=1}^n\cos(a_i t)-\prod_{i=2}^{n}\frac{1}{1+a_{i}^2t^2/2}\geq -a_1^2t^2/2,
\]
thus
\[
\Ex|S_1|^p-\Ex|S_2|^p=
C_p\int_{0}^{\infty}\Big(\varphi_{S_1}(t)-\varphi_{S_2}(t)+a_1^2t^2/2\Big)t^{-p-1}dt
\geq 0.
\]
\end{proof}

To generalize the above result to arbitrary $p>2$ we need one more easy estimate.

\begin{lem}
For any real numbers $a,b$ we have
\begin{equation}
\label{rec1}
\Ex|a\cale+b|^p= |b|^p+\frac{p(p-1)}{2}a^2\Ex|a\cale+b|^{p-2} \mbox{ for } p\geq 2
\end{equation}
and
\begin{equation}
\label{rec2}
\Ex|a\ve+b|^p\geq |b|^p+\frac{p(p-1)}{2}a^2|b|^{p-2} \mbox{ for } p\geq 3.
\end{equation}
\end{lem}

\begin{proof}
By integration by parts it is easy to show that for any $f\in C^{2}(\er)$ of at most
polynomial growth we have $\Ex f(\cale)=f(0)+\frac{1}{2}\Ex f''(\cale)$. If we take $f(x)=|ax+b|^p$
we obtain (\ref{rec1}). To prove (\ref{rec2}) it is enough to notice that the function
$g(x):=\Ex|x\ve+b|^p$ satisfies $g(0)=|b|^p$, $g'(0)=0$ and 
$g''(x)=p(p-1)\Ex|x\ve+b|^{p-2}\geq p(p-1)|b|^{p-2}$. 
\end{proof}

Our first theorem shows that moments of Bernoulli sums dominate moments of exponential sums up to few largest coefficients.

\begin{thm}
\label{rad_exp}
Let $|a_1|\geq |a_2|\geq \ldots \geq |a_n|$. Then for any $p\geq 2$,
\begin{equation}
\label{comp2}
\gamma_p^p\Big(\sum_{i=1}^n a_i^2\Big)^{p/2}\geq 
\Ex\Big|\sum_{i=1}^n a_{i}\ve_i\Big|^{p}\geq
\Ex\Big|\sum_{i=\lceil p/2\rceil}^na_{i}\cale_i\Big|^{p}
\geq \gamma_p^p\Big(\sum_{i=\lceil p/2\rceil}^n a_i^2\Big)^{p/2}.
\end{equation}
\end{thm}

\begin{proof}
To establish the middle inequality we will show by double induction first on $k$ then on 
$n$ that for $p\in (2k,2k+2]$,
\begin{equation}
\label{comp3}
\Ex\Big|\sum_{i=1}^n a_{i}\ve_i\Big|^{p}\geq
\Ex\Big|\sum_{i=k+1}^na_{i}\cale_i\Big|^{p}.
\end{equation}
For $k=1$ this follows by Lemma \ref{p_2_4}. Suppose that our assertion holds for $k-1$ and
let $p\in (2k,2k+2]$. For $n<k+1$ the inequality (\ref{comp2}) is obvious. If $n\geq k+1$
and (\ref{comp3}) holds for $n-1$ then by (\ref{rec2}), induction assumption, 
and (\ref{rec1}),
\begin{align*}
\Ex\Big|\sum_{i=1}^n a_{i}\ve_i\Big|^{p}
&\geq
\Ex\Big|\sum_{i=2}^n a_{i}\ve_i\Big|^{p}+a_1^2\frac{p(p-1)}{2}
\Ex\Big|\sum_{i=2}^n a_{i}\ve_i\Big|^{p-2}
\\
&\geq
\Ex\Big|\sum_{i=k+2}^n a_{i}\cale_i\Big|^{p}
+a_{k+1}^2\frac{p(p-1)}{2}\Ex\Big|\sum_{i=k+1}^n a_{i}\ve_i\Big|^{p-2}
\\
&=\Ex\Big|\sum_{i=k+1}^n a_{i}\cale_i\Big|^{p}.
\end{align*}

First inequality in (\ref{comp2}) follows by the Khintchine inequality with optimal
constant \cite{Ha} and
the last inequality in (\ref{comp2}) is an easy consequence of the fact that $\cale$ is
a mixture of gaussian r.v.'s (see Remark 5 in \cite{KLO}).
\end{proof}

Next two corollaries present more precise versions of inequalities (\ref{radmom}) and (\ref{expmom}).

\begin{cor}
\label{estrad}
For any $p\geq 2$ we have
\begin{align*}
\max\bigg\{\gamma_p\bigg(\sum_{i\geq \lceil p/2\rceil} (a_{i}^*)^2\bigg)^{1/2},
\frac{1}{\sqrt{2}}\sum_{i<\lceil p/2\rceil}&a_i^*\bigg\}
\leq \Big\|\sum_{i=1}^n a_{i}\ve_i\Big\|_{p}
\\
&\leq \gamma_p\bigg(\sum_{i\geq \lceil p/2\rceil} (a_{i}^*)^2\bigg)^{1/2}+
\sum_{i<\lceil p/2\rceil}a_i^*.
\end{align*}
\end{cor}

\begin{proof}
We have by the triangle inequality and the Khintchine inequality with optimal constant
\cite{Ha},
\begin{align*}
\Big\|\sum_{i=1}^n a_{i}\ve_i\Big\|_{p}&\leq
\Big\|\sum_{i\geq \lceil p/2\rceil} a_{i}^*\ve_i\Big\|_{p}
+\Big\|\sum_{i< \lceil p/2\rceil} a_{i}^*\ve_i\Big\|_{p}
\\
&\leq \gamma_p\Big(\sum_{i\geq \lceil p/2\rceil} (a_{i}^*)^2\Big)^{1/2}+
\sum_{i<\lceil p/2\rceil}a_i^*.
\end{align*}
To show the lower bound we use (\ref{comp2})
\[
\Big\|\sum_{i=1}^n a_{i}\ve_i\Big\|_{p}=\Big\|\sum_{i=1}^n a_{i}^*\ve_i\Big\|_{p}
\geq \gamma_p\Big(\sum_{i\geq \lceil p/2\rceil} (a_{i}^*)^2\Big)^{1/2}
\]
and an easy estimate
\[
\Big\|\sum_{i=1}^n a_{i}\ve_i\Big\|_{p}\geq 
\Big\|\sum_{i<\lceil p/2\rceil} a_{i}^*\ve_i\Big\|_{p}\geq
\big(\Pr(\ve_{i}=1 \mbox{ for }1\leq i< \lceil p/2\rceil)\big)^{1/p}\sum_{i<\lceil p/2\rceil}a_i^*.
\]

\end{proof}

\begin{cor}
\label{estexp}
For any $p\geq 2$ we have
\[
\max\Big\{\gamma_p\|a\|_2,\frac{p}{e\sqrt{2}}\|a\|_{\infty}\Big\}
\leq \Big\|\sum_{i=1}^n a_{i}\cale_i\Big\|_{p}
\leq \gamma_p\|a\|_2+p\|a\|_{\infty}.
\]
\end{cor}

\begin{proof}
Let $S=\sum_{i=1}^n a_{i}\cale_i$ and  $k=\lceil p/2\rceil-1$. We have $\|S\|_p\geq \gamma_p\|a\|_2$ by the last inequality in 
(\ref{comp2}). Moreover
\[
\|S\|_p\geq \|a\|_{\infty}\|\cale\|_p=\|a\|_{\infty}\frac{1}{\sqrt{2}}(\Gamma(p+1))^{1/p}
\geq \frac{p}{\sqrt{2}e}\|a\|_{\infty}.
\]
To get the upper bound we use twice bounds (\ref{comp2})
and obtain
\begin{align*}
\|S\|_p-\gamma_p\|a\|_2&\leq \|S\|_p-\Big\|\sum_{i>k}a_i^*\cale_i\Big\|_p
\leq \Big\|\sum_{i\leq k}a_i^*\cale_i\Big\|_p\leq 
\|a\|_{\infty}\Big\|\sum_{i\leq k}\cale_i\Big\|_p
\\
&\leq \|a\|_{\infty}\Big\|\sum_{i\leq 2k}\ve_i\Big\|_p\leq 
2k\|a\|_{\infty}\leq p\|a\|_{\infty}.
\end{align*}
\end{proof}

Now we may state a result that generalizes (up to a multiplicative constant)
previous corollaries.

\begin{thm}
\label{estlogconc}
Let $X_i$ be independent symmetric r.v.'s with logarithmically concave tails such that
$\Ex X_i^2=1$ and $|a_1|\geq |a_2|\geq \ldots\geq |a_n|$. Then for any $p\geq 3$,
\begin{align*}
\max\bigg\{\gamma_p\bigg(\sum_{i\geq \lceil p/2\rceil} a_{i}^2\bigg)^{1/2},
\Big\|\sum_{i<p}a_iX_i\Big\|_p\bigg\}&
\leq \Big\|\sum_{i=1}^n a_{i}X_i\Big\|_{p}
\\
&\leq \gamma_p\bigg(\sum_{i\geq \lceil p/2\rceil} a_{i}^2\bigg)^{1/2}+
\Big\|\sum_{i<p}a_iX_i\Big\|_p.
\end{align*}
\end{thm}

\begin{proof}
Lower bound is an immediate consequence of Theorem \ref{rad_exp} and Proposition \ref{extr}.
To get the upper bound let $k=\lceil p/2\rceil -1$. Then
\begin{align*}
\Big\|\sum_{i=1}^n a_{i}X_i\Big\|_{p}&\leq
\Big\|\sum_{i>2k}a_{i}X_i\Big\|_{p}+
\Big\|\sum_{i\leq 2k}a_{i}X_i\Big\|_{p}
&\leq \gamma_p\Big(\sum_{i>k} a_{i}^2\Big)^{1/2}+
\Big\|\sum_{i\leq 2k}a_iX_i\Big\|_p
\end{align*}
again by Theorem \ref{rad_exp} and Proposition \ref{extr}.
\end{proof}

{\bf Remark.} By the result of Gluskin and Kwapie\'n we have
\[
\Big\|\sum_{i<p}a_iX_i\Big\|_p\sim 
\sup\Big\{\sum_{i<p}a_ib_i\colon \sum_{i<p}M_i(b_i)\leq p\Big\},
\]
where $M_i(x)=x^2$ for $|x|\leq 1$ and $M_i(x)=-\ln\Pr(|X_i|\geq x)$ for $|x|>1$. 

\medskip

We conclude with one more result about Gaussian approximation of moments.

\begin{cor}
Let $X_i$ be as in Theorem \ref{estlogconc}, then for any $p\geq 3$,
\[
\Big| \Big\|\sum_{i=1}^n a_{i}X_i\Big\|_{p}-\gamma_p\|a\|_2\Big|\leq
p\|a\|_{\infty}.
\]
\end{cor}

\begin{proof}
The statement immediately follows by Proposition \ref{extr} and Corollaries \ref{estrad} 
and \ref{estexp}.
\end{proof}

\end{document}